\documentclass[12pt,a4paper]{amsart}
\usepackage{amssymb,amsthm,amsmath,currfile,pdfpages}
\usepackage{geometry}
\usepackage{amssymb,color}
\usepackage[english]{babel}
\usepackage{verbatim,here}
\usepackage[T1]{fontenc}
\usepackage{graphicx}
\usepackage{epstopdf}
\usepackage{floatflt}
\usepackage{color,xcolor}
\usepackage{enumerate}
\usepackage{bm}
\usepackage{hyperref}
\usepackage{animate}
\usepackage{a4wide}
\usepackage{amsmath, amssymb}
\usepackage{tikz}

\font\fFt=eusm10 
\font\fFa=eusm7  
\font\fFp=eusm5  
\def\K{\mathchoice
{\hbox{\,\fFt K}}
{\hbox{\,\fFt K}}
{\hbox{\,\fFa K}}
{\hbox{\,\fFp K}}}

\newcommand{\M}{\mathsf{M}}

\font\fFt=eusm10 
\font\fFa=eusm7  
\font\fFp=eusm5  

\makeatletter
\renewcommand{\subsection}{%
    \stepcounter{subsection}
    \addtocounter{equation}{+1}
    \setcounter{subsection}{\value{equation}}
    \bigskip
    \noindent{{\bfseries \arabic{section}.\arabic{subsection}.\ }}}
\renewcommand{\thesubsubsection}{%
    \@arabic\c@section.\@arabic\c@subsection.\@arabic\c@subsubsection}
\makeatother

\newcounter{minutes}
\divide\time by 60
\newcounter{hours}
\multiply\time by 60 \addtocounter{minutes}{-\time}


\dedicatory{}
\commby{}
\theoremstyle{plain}
\newtheorem{theorem}{Theorem}[section]

\theoremstyle{definition}
\newtheorem{definition}{Definition}[section]
\newtheorem{lemma}[equation]{Lemma}
\newtheorem{example}{Example}[section]

\theoremstyle{remark}
\newtheorem{remark}{Remark}[section]

\theoremstyle{definition}

\theoremstyle{remark}

\numberwithin{equation}{section}

\newcommand{\beq}{\begin{equation}}
\newcommand{\eeq}{\end{equation}}
\newcommand{\ben}{\begin{enumerate}}
\newcommand{\een}{\end{enumerate}}
\newcommand{\bequu}{\begin{eqnarray*}}
\newcommand{\eequu}{\end{eqnarray*}}
\newcommand{\bequ}{\begin{eqnarray}}
\newcommand{\eequ}{\end{eqnarray}}

\renewcommand{\thefootnote}{\number_style{footnote}}
\begin{document}
\def\thefootnote{}

\title [Harnack inequality and Schwarz lemma]{On Harnack inequality and harmonic Schwarz lemma}

\author[R. Kargar]{Rahim Kargar}
\address{Department of Mathematics and Statistics, University of Turku,
         Turku, Finland}
\curraddr{} \email{rakarg@utu.fi;
DOI: \url{http://orcid.org/0000-0003-1029-5386}} \curraddr{}

\date{}

\begin{abstract}
In this paper, we study the $(s, C(s))$-Harnack inequality in a domain $G\subset \mathbb{R}^n$ for $s\in(0,1)$ and $C(s)\geq1$ and present a series of inequalities related to $(s, C(s))$-Harnack functions and the Harnack metric. We also investigate the behavior of the Harnack metric under $K$-quasiconformal and $K$-quasiregular mappings, where $K\geq 1$. Finally, we provide a type of harmonic Schwarz lemma and improve the Schwarz-Pick
estimate for a real-valued harmonic function.
\end{abstract}

\maketitle

\footnotetext{\texttt{{\tiny File: Harnack.tex, printed: \number\year-%
\number\month-\number\day, \thehours.\ifnum\theminutes<10{0}\fi\theminutes}}}
\makeatletter

\makeatother


\section{Introduction}\label{sec1-Intro}
Harnack's inequality is a fundamental result in the study of partial differential equations (PDEs), with applications across various branches of mathematics, particularly in the theory of elliptic and parabolic equations. The Harnack inequality typically concerns positive solutions to elliptic or parabolic equations in divergence form. In the case of elliptic equations, which describe steady-state problems such as heat conduction or electrostatics, the inequality establishes bounds on the solutions by comparing the maximum and minimum values within a domain. Moreover, the German mathematician Axel Harnack developed the original formulation of this inequality for harmonic functions in the plane, see \cite{Ka} for more details. It should be noted that this inequality was first published in 1887 in the book \cite{hr}. 

In the context of the theory of partial differential equations, the current formulation of the Harnack inequality for harmonic functions is expressed as follows:\\
\noindent
{\bf Harnack inequality. }
Let $B^n(x,r)$ be a Euclidean ball centered at $x$ with the radius $r\in(0,1)$ such that the concentric ball $B^n(x,2r)$ is contained in a domain $G \subset \mathbb{R}^n$, $n\geq 2$. Then there exists a positive constant $C$ depending on $n$ such that
\begin{equation}\label{inq-Harnack}
  \sup_{B^n(x,r)} u(z)\leq C \inf_{B^n(x,r)} u(z)
\end{equation}
holds for all nonnegative harmonic functions $u:G\rightarrow\mathbb{R}$.

We recall that a real-valued function $u: G\subset \mathbb{R}^n \rightarrow\mathbb{R}$ is called harmonic in a domain $G \subset \mathbb{R}^n$ if it is twice continuously differentiable and satisfies the Laplace equation $\sum_{i=1}^{n}\partial^2 u/\partial x_i^2=0$.
The progress of potential analysis linked to the Laplace equation hinges on the key role of Harnack's inequality \eqref{inq-Harnack}, see \cite{he}.

Subsequently, we revisit a definition presented in \cite{v1}. Define $\mathbb{R}^+$ as the set $\{x\in \mathbb{R}:x>0\}$.
\begin{definition}\label{def-Harnack-HKV}
Consider a proper subdomain $G$ of $\mathbb{R}^n$, and let $u:G\rightarrow \mathbb{R}^+\cup\{0\}$ be a continuous function. We say that $u$ satisfies the Harnack inequality in $G$ if there exist numbers $s\in(0,1)$ and $C(s)\geq 1$ such that
\begin{equation}\label{inq-Harnack s}
  \max_{B_x}\,\, u(z)\leq C(s) \min_{B_x}\,\, u(z)
\end{equation}
holds, whenever $B^n(x,r)\subset G$ and $B_x=\overline{B}^n(x,sr)$. A function that meets this condition is referred to as a Harnack function.
\end{definition}
Here are some examples:
\begin{example}
i) Let $u:G\rightarrow \mathbb{R}^+$ be a continuous function on a domain $G$ with $0<m\leq u(x)\leq M<\infty$. Then $u$ satisfies \eqref{inq-Harnack s} with $C(s)=M/m$ for all $x\in G$.\\
ii) Let $G$ be a domain and $d(x, \partial G)$ be the minimum distance from $x$ to the boundary of $G$. If $u:G\rightarrow \mathbb{R}^+$ is defined as $u(x)=\alpha\,d(x, \partial G)^\beta$, where $\alpha>0$ and $\beta\neq0$, then $u$ satisfies \eqref{inq-Harnack s} with $C(s)=((1+s)/(1-s))^{|\beta|}$.\\
iii) All nonnegative harmonic functions satisfy \eqref{inq-Harnack s} with a constant $C(s)$ such that $C(s)\rightarrow 1$ as $s\rightarrow 0^+$, see \cite[p. 16]{gt}.\\
iv) Let $u(z)=\arg z$ and $G=\mathbb{R}^2\setminus \{x\in \mathbb{R}: x\geq 0\}$. Then $u$ satisfies \eqref{inq-Harnack s} in $G$ with $C(s)=(4+\pi)/(4-\pi)$, where $s=1/2$; see \cite[Exercise 6.33(1)]{hkv}.
\end{example}

In this paper, we study the $(s, C(s))$-Harnack inequality, which is defined as follows, where $s\in(0,1)$ and $C(s)\geq1$.
\begin{definition}\label{def-sC harnack}
 Under the assumptions of Definition \ref{def-Harnack-HKV}, for $s\in(0,1)$ and $C_s\geq1$ we say that $u$ satisfies the $(s, C(s))$-Harnack inequality in a domain $G\subset \mathbb{R}^n$, if the inequality \eqref{inq-Harnack s} holds. A function satisfying \eqref{inq-Harnack s} for all $s\in(0,1)$ is called the $(s,C(s))$-Harnack function.
\end{definition}
This paper is organized as follows: Section \ref{sec2-Pre} provides the essential notations and definitions required for the discussions in this paper. In Section \ref{sec3-Main Results}, we investigate the behavior of the $(s, C(s))$-Harnack functions and the Harnack metric. Lastly, Section \ref{sec4-inequ} presents a version of the harmonic Schwarz lemma and improves the Schwarz-Pick estimate for a real-valued harmonic function.

\section{Preliminaries}\label{sec2-Pre}
This section establishes a foundation for our subsequent discussions by introducing essential notations and definitions.

Let sh, ch, th, and arth denote the hyperbolic functions $\sinh$, $\cosh$, $\tanh$, and ${\rm arctanh}$ respectively.
Consider the Euclidean space $\mathbb{R}^n$ with $n \geq 2$ and define $\mathbb{H}^n = \{x=(x_1, \ldots, x_n) \in \mathbb{R}^n : x_n > 0\}$ as the Poincar\'{e} half-space or the upper half-plane. The ball with center $x$ in $\mathbb{R}^n$ and radius $r>0$ is denoted as $B^n(x,r)$, defined as the set $\{y\in \mathbb{R}^n:|y-x|<r\}$. Correspondingly, the sphere sharing the same center and radius is $S^{n-1}(x,r)=\{y\in \mathbb{R}^n:|y-x|=r\}$. The unit ball will be denoted by $\mathbb{B}^n=B^n(0,1)$. Also, $\overline{B}^n(x,r)=\{y\in \mathbb{R}^n:|y-x|\leq r\}$. For any
$x$ within a domain $G$ in $\mathbb{R}^n$, the Euclidean distance $d_G(x)$ is defined as the minimum distance from $x$ to the boundary of $G$, denoted by $d_G(x)=d(x,\partial G)=\inf \{|x-w|: w\in \partial G\}$. In the hyperbolic space $\mathbb{H}^n$, the hyperbolic distance $\rho$ is characterized by the differential ${\rm d}\rho=|{\rm d}x|/x_n$. Explicit formulas for the distances between points in both the upper half-space $\mathbb{H}^n$ and the unit ball $\mathbb{B}^n$, respectively, are as follows (see \cite[(4.8), p. 52; (4.16), p. 55]{hkv}):
\begin{equation*}
  {\rm ch} \rho_{\mathbb{H}^n}(x,y)=
  1+\frac{|x-y|^2}{2d_{\mathbb{H}^n}(x)d_{\mathbb{H}^n}(y)},\quad x,y\in \mathbb{H}^n,
\end{equation*}
and
\begin{equation*}
  {\rm sh}^2 \frac{\rho_{\mathbb{B}^n}(x,y)}{2}=\frac{|x-y|^2}{(1-|x|^2)(1-|y|^2)},\quad
  x,y\in \mathbb{B}^n.
\end{equation*}
The quasihyperbolic distance, denoted as $k_G(x, y)$, between points $x$ and $y$ in the domain $G$, is formally defined as the infimum of the integral along rectifiable curves $\gamma \subset G$ containing both $x$ and $y$. This integral is calculated as the quotient of the absolute value of the differential element ${\rm d}x$ by the distance function $d_G(x)$, as given by the expression:
\begin{equation*}
  k_G(x,y)=\inf_{\gamma}
  \int_{\gamma}\frac{|{\rm d}x|}{d_G(x)}.
\end{equation*}
Gehring and Palka introduced the metric $k_G(x,y)$ in \cite[p. 173]{gp} and provided a proof for the sharp inequalities (\cite[Lemma 2.1]{gp}). These inequalities are expressed as follows:
\begin{equation}\label{kG geq abs log}
k_G(x,y)\geq \left|\log \frac{d_G(x)}{d_G(y)}\right|
\end{equation}
and
\begin{equation}\label{kG geq log}
k_G(x,y)\geq \log \left(1+\frac{|x-y|}{d_G(x)}\right).
\end{equation}
For a detailed discussion, we refer to \cite[p. 68]{hkv}. It is well-known that (see \cite[p. 174]{gp})
\begin{equation}\label{inq-pkp}
k_{\mathbb{H}^n}(x,y)=\rho_{\mathbb{H}^n}(x,y),\quad {\rm and}\quad k_{\mathbb{B}^n}(x,y) \leq \rho_{\mathbb{B}^n}(x,y) \leq 2k_{\mathbb{B}^n}(x,y).
\end{equation}
For any open set $\Omega$ in $\mathbb{R}^n$, where $\Omega$ is not equal to the entire space $\mathbb{R}^n$, the distance ratio metric is defined by
\begin{equation*}
j_\Omega(x, y)=\log\left(1+\frac{|x-y|}{\min\{d_\Omega(x),d_\Omega(y)\}} \right),\quad x, y\in \Omega.
\end{equation*}
When $\Omega\in\{\mathbb{B}^n, \mathbb{H}^n\}$ as per \cite[Lemma 4.9]{hkv}, the following double-inequality holds:
\begin{equation}\label{inq-jpj}
j_\Omega(x,y)\leq \rho_{\Omega}(x,y)\leq 2 j_\Omega(x,y).
\end{equation}
{\bf Modulus of a curve family.}
Let $\Gamma$ be a family of curves in $\mathbb{R}^n$. Also, let $\mathcal{F}(\Gamma)$ denote the family of all non-negative Borel-measurable functions $\sigma:\mathbb{R}^n\to \mathbb{R} \cup \{\infty\}$ such that $\int_\gamma\sigma {\rm d}\tau\geq1$ for each locally rectifiable curve $\gamma\in\Gamma$. The modulus of a curve family $\Gamma\subset\mathbb{R}^n$ is defined by (see \cite[p. 104]{hkv})
\begin{equation*}
\M(\Gamma)=\inf_{\sigma\in\mathcal{F}(\Gamma)}\int_{\mathbb{R}^n}\sigma^n {\rm d} m,
\end{equation*}
where $m$ stands for the $n$-dimensional Lebesgue measure.

We denote by $\Delta(E,F;G)$ the family of all closed non-constant curves joining two non-empty sets $E$ and $F$ in a domain $G$, where $E$, $F$, and $G$ are subsets of $\overline{\mathbb{R}}^n$.

\noindent
{\bf Modulus metric.}
Let $G$ be a proper subdomain of $\overline{\mathbb{R}}^n$. The modulus metric is defined by
\begin{equation*}
\mu_G(x,y)=\inf_{C_{xy}}\M(\Delta(C_{xy},\partial G;G)),
\end{equation*}
where the infimum is taken over all continuous paths $C_{xy}$ in $G$ joining $x$ and $y$, represented by a continuous function $\gamma:[0,1]\rightarrow G$ satisfying $\gamma(0)=x$ and $\gamma(1)=y$. The definition of modulus metric is illustrated in Figure \ref{Fig: Def mu metric}.
\begin{figure}
    \centering
\begin{tikzpicture}
\draw [thick] plot [smooth cycle] coordinates
{(0,0)(-1,-1.4)(-3,0)(-2.5,1)(-2, 2)(-1,2)(1,4)(4,2)(4,1)(2,1)(1,-0.5)};

\node [black] at (-2, -0.03) {\textbullet}; 
\node[scale=1] at (-2, 0.4) {$x$}; 

\node [black] at (2, 2.99) {\textbullet}; 
\node[scale=1] at (2.2, 2.6) {$y$}; 

\draw [thick, dashed] (0, 3.2) .. controls (1.5,2) and (-1.5,2.2) .. (0.2,0.8); 
\node[scale=1] at (3, 0.5) {$G$}; 
\node[scale=1] at (1, 1) {$C_{xy}$}; 
\draw[thick] (-2,0) .. controls (-0.5,1.8) and (0,-0.6) .. (0.3,0.6) .. controls (1,3) and (0.2,0) .. (2,3);
\end{tikzpicture}
\caption{Conformal invariant $\mu_G(x,y)$}
\label{Fig: Def mu metric}
\end{figure}
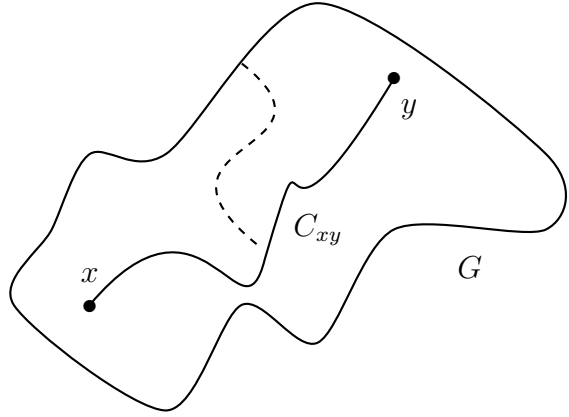

\noindent
{\bf Uniformity. } (See \cite[Definition 6.1]{hkv})
A domain $G$ of $\mathbb{R}^n$, where $G\neq \mathbb{R}^n$, is termed {\it uniform} if there exists a constant $A=A(G)\geq1$ such that $k_G(x,y)\leq A j_G(x,y)$ for all $x,y\in G$. The unit ball $\mathbb{B}^n$ and the upper half-space $\mathbb{H}^n$ are examples of uniform domains with the constant $2$, as implied by \eqref{inq-pkp} and \eqref{inq-jpj}, respectively.

\noindent
{\bf Absolutely Continuous on Lines (ACL).} Consider $\mathbb{R}^{n-1}_j$ as the set $\mathbb{R}^{n-1}_j=\{x\in\mathbb{R}^n: x_j=0\}$, where $j=1,2,\ldots,n$. Suppose that $T_j:\mathbb{R}^n\rightarrow \mathbb{R}^{n-1}_j$ is an onto orthogonal projection $T_j x=x-x_je_j$ and $Q=\{x\in\mathbb{R}^n: a_j\leq x_j\leq b_j\}$ is a closed $n$-interval. A mapping $\phi: Q \rightarrow \mathbb{R}$ is called {\it absolutely continuous on lines}, abbreviated as ACL, if it is absolutely continuous on almost every line segment in $Q$, parallel to the coordinate axes $e_1,\ldots,e_n$. More precisely, if $E_j$ is the set of all $x\in T_j Q$ such that the mapping $t\mapsto \phi(x+te_j)$ is not absolutely continuous on $[a_j,b_j]$, then $m_{n-1}(E_j)=0$ for all $j=1,\ldots,n$.

For an open set $\Omega$ in $\mathbb{R}^n$, an ACL mapping $\phi:\Omega\rightarrow \mathbb{R}$ is said to be ACL$^n$, $n\geq1$, if $\phi$ is locally $L^n$-integrable in $\Omega$ and if the partial derivatives $\partial_j \phi$ (which exist a.e. and are measurable) of $\phi$ are locally $L^n$-integrable as well; see Ref. \cite[p. 22]{mrsy}.

\noindent
{\bf Quasiregular mappings. }
Consider a domain $G\subset \mathbb{R}^n$. A mapping $f: G \rightarrow \mathbb{R}^n$ is said to be $K$-quasiregular if $f$ belongs to ACL$^n$ and if there exists a constant $K\geq1$ satisfying the inequality
\begin{equation*}
|f'(x)|^n\leq K J_f(x),\quad {\rm where}\quad|f'(x)|=\max_{|\phi|=1}|f'(x)\phi|,
\end{equation*}
almost everywhere in $G$. Here, $f'(x)$ and $J_f(x)$ represent the formal derivative and the Jacobian determinant of $f$ at the point $x$, respectively.

\noindent
{\bf Quasiconformal mappings.}
Let $G$, $G'$ be domains in $\overline{\mathbb{R}}^n=\mathbb{R}^n\cup \{\infty\}$, $K\geq1$ and let $f:G\rightarrow G'$ be a homeomorphism. Then, $f$ is $K$-quasiconformal if and only if the following conditions are satisfied:
\begin{itemize}
    \item $f$ is ACL$^n$;
    \item $f$ is differentiable;
    \item for almost all $x\in G$
    \begin{equation*}
        |f'(x)|^n/K\leq |J_f(x)|\leq K L (f'(x))^n,
    \end{equation*}
where $L(\lambda)=\min_{|\phi|=1}|\lambda\phi|$.
\end{itemize}
The Harnack inequality provides a basis for defining a Harnack (pseudo) metric. Consider $\mathcal{H}^+(G)$ as the class of all positive harmonic functions $u$ in $G$.

\noindent
{\bf Harnack metric. }
For arbitrary $x, y \in G$, the Harnack metric is defined by
\begin{equation*}
h_G(x,y)=\sup\left|\log \frac{u(x)}{u(y)}\right|,
\end{equation*}
where the supremum is taken over all $u\in \mathcal{H}^+(G)$. This metric has been investigated in various contexts, including studies in \cite{bs, c, hn, hm, k, s}.

\section{$(s, C(s))$-Harnack functions and Harnack metric}\label{sec3-Main Results}
In this section, we present our results on $(s, C(s))$-Harnack functions and the Harnack metric under $K$-quasiconformal and $K$-quasiregular mappings. We start with the following:
\begin{lemma}{\it
All positive harmonic functions on $B^n(x,r)\subset\mathbb{R}^n$ are $(s, C(s))$-Harnack with
\begin{equation*}
    C(s)=C(s,n)=\frac{1}{1-s^2}\left(\frac{1+s}{1-s}\right)^n
\end{equation*}
for all $s\in(0,1)$.}
\end{lemma}
\begin{proof}
Let $u$ be any positive harmonic function on $B^n(x,r)$ and $0<\delta<r$. Then, by  \cite[Theorem 3.2.1]{he} we have
\begin{equation}\label{the-Helms}
    \frac{u(x_1)}{u(x_2)}\leq \frac{r^2}{r^2-\delta^2}\left(\frac{r+\delta}{r-\delta}\right)^n
\end{equation}
for all $x_1,x_2\in B^n(x,\delta)$. Now, it is enough to put $\delta=r s$ in \eqref{the-Helms} since $rs<r$ for all $s\in(0,1)$.
\end{proof}
\begin{theorem}\label{Th-Har inq-i-ii}
(i) Let $s\in(0,1)$ and $u:\mathbb{B}^{n} \rightarrow (0,\infty)$ be a Harnack function. Then for all $x,y\in \mathbb{B}^n$
\begin{equation*}
u(x)\leq C(s)^{1+t} u(y), \qquad t=\frac{\log((1+r)/(1-r))}{\log((1+s)/(1-s))},
\end{equation*}
where $r={\rm th}(\rho_{\mathbb{B}^n}(x,y)/2)$ and $C(s)\geq 1$.\\
(ii) If $u$ is a positive harmonic function, $x\in \mathbb{B}^n$, $s\in(0,1)$ and $y\in S^{n-1}(x, s(1-|x|))$, then
\begin{equation*}
  u(x)\leq \frac{1}{1-s^2}\left(\frac{1+s}{1-s}\right)^n u(y), \qquad s<\exp(\rho_{\mathbb{B}^n}(x,y))-1.
\end{equation*}
\end{theorem}
\begin{proof}
(i) The proof follows from Definition \ref{def-sC harnack} and \cite[Lemma 6.23]{hkv}.\\
(ii)  It follows from \cite[Lemma 4.9(1)]{hkv} that
\begin{equation*}
  \rho_{\mathbb{B}^n}(x,y)\geq j_{\mathbb{B}^n}(x,y)\geq \log\left(1+\frac{s(1-|x|)}{1-|x|}\right)=\log(1+s).
\end{equation*}
This completes the proof.
\end{proof}

We continue with the following result on quasiregular mappings; in fact, we show that if $f: G\rightarrow \mathbb{R}^n$ is a quasiregular mapping, and if $\partial fG$ satisfies some additional conditions, then the function $u(x)=d_{fG}(f(x))$, $(x\in G)$, satisfies the $(s, C(s))$-Harnack inequality.
\begin{remark}
It is important to clarify that the theorem presented herein diverges from Theorem 5.2 in \cite{svz}. Specifically, our theorem assumes that $fG$ is a $A$-uniform domain with a connected boundary, while Sugawa et al. \cite{svz} regarded $\partial fG$ as uniformly perfect. The connectedness of $\partial fG$ is decisive in the following theorem, as demonstrated in Remark \ref{rem-exa} below. Conversely, in the proof of Theorem 5.2, Sugawa et al. \cite{svz} employ the definition of the modulus metric $\mu_G$ to establish an upper bound, whereas we utilize a general upper bound derived from Lemma 10.6(2) of \cite{hkv} for $y\in {B}^n(x,s d_G(x))$. Moreover, the constant $C(s)$ obtained here is more generality than the constant obtained by Sugawa et al. in \cite{svz}.
\end{remark}

\begin{theorem}\label{Thm-Har con for qr}
Let $G$ be a proper subdomain of $\mathbb{R}^n$, and $f: G\rightarrow \mathbb{R}^n$ be a $K$-quasiregular mapping such that $fG\subset \mathbb{R}^n$ is a $A$-uniform domain. Also, let $\partial fG$ be connected such that it consists of at least two points. Then, the function $u(x)=d_{fG}(f(x))$, $(x\in G)$, satisfies the $(s, C(s))$-Harnack inequality with the constant
\begin{equation}\label{abs fx leq C abs fy}
C(s)=\exp\left(\frac{A K_I(f)}{c_n}\omega_{n-1}\left(\log \frac{s d_G(x)}{|x-y|}\right)^{1-n}\right),\quad s\in(0,1),
\end{equation}
for $y\in B_{x,s}={B}^n(x,s d_G(x))$, where $\omega_{n-1}$ is the $(n-1)$-dimensional surface area of $S^{n-1}$, $K_I(f)$ is the inner dilatation of $f$, and $c_n$ is a constant number depending only on $n$.
\end{theorem}

\begin{proof}
Since $\partial fG$ is a connected domain and $fG$ is a $A$-uniform domain, by \cite[Lemma 10.8(1)]{hkv} and by definition, we have
\begin{equation}\label{mu fG geq cn J geq cnA}
  \mu_{fG}(f(x),f(y))\geq c_n j_{fG}(f(x),f(y))\geq \frac{c_n}{A}k_{fG}(f(x),f(y)),\quad x,y\in G,
\end{equation}
where $A\geq 1$, and $c_n$ is a constant number depending on $n$.
Also, by \cite[Theorem 15.36(1)]{hkv} the following inequality
\begin{equation}\label{inq-f qr}
  \mu_{fG}(f(x),f(y))\leq K_I(f)\mu_G(x,y),\quad x,y\in G
\end{equation}
holds for a non-constant quasiregular mapping $f: G\rightarrow \mathbb{R}^n$, where $K_I(f)\geq1$ is the inner dilatation
of $f$. It follows from \cite[Lemma 10.6(2)]{hkv} that if $x\in G$ and $y\in B_{x,s}={B}^n(x,s d_G(x))$ with $x\neq y$, then
\begin{equation}\label{inq- mu G}
  \mu_G(x,y)\leq \mu_{B_{x,s}}(x,y)\leq \omega_{n-1}\left(\log \frac{1}{r} \right)^{1-n},
\end{equation}
where $r=|x-y|/(sd(x))$. Now, by \eqref{kG geq abs log} and  \eqref{mu fG geq cn J geq cnA}-\eqref{inq- mu G}, we obtain
\begin{align*}
  \left|\log \frac{d_{fG}(f(x))}{d_{fG}(f(y))}\right| &\leq k_{fG}(f(x),f(y))\leq \frac{A}{c_n}\mu_{fG}(f(x),f(y))\leq  \frac{A K_I(f)}{c_n}\mu_{G}(x,y) \\
  &\leq  \frac{AK_I(f)}{c_n}\mu_{B_{x,s}}(x,y)\leq \frac{AK_I(f)}{c_n}\omega_{n-1}\left(\log \frac{s d(x)}{|x-y|}\right)^{1-n}.
\end{align*}
This establishes the desired inequality \eqref{abs fx leq C abs fy}, and thus concludes the proof.
\end{proof}
\begin{remark}\label{rem-exa}
In Theorem \ref{Thm-Har con for qr}, the connectedness of $\partial fG$ is crucial. However, it is noteworthy that the statement of Theorem \ref{Thm-Har con for qr} can be invalidated by the existence of an analytic function $f:\mathbb{B}^2\rightarrow \mathbb{B}^2\setminus \{0\}=f\mathbb{B}^2$. An explicit example of such a function is defined by $f:\mathbb{B}^2\rightarrow \mathbb{B}^2\setminus \{0\}$ as
\begin{equation*}
f(z)=\exp\left(\frac{z+1}{z-1}\right),\quad z\in \mathbb{B}^2.
\end{equation*}
Let $x_p = (e^p - 1)/(e^p + 1)$ for $p = 1, 2, \ldots$. Considering $f(x_p) = \exp(-e^p)$ and $f(x_{p+1}) = \exp(-e^{p+1})$, we can deduce
\begin{equation*}
\left|\frac{f(x_p)}{f(x_{p+1})}\right| = \frac{\exp(e^{p+1})}{\exp(e^p)}.
\end{equation*}
Additionally, employing a straightforward calculation, we can infer from \eqref{kG geq log} that
\begin{equation*}
\left(\log \frac{sd(x_p)}{|x_p - x_{p+1}|}\right)^{1-n} \leq \left(\log \frac{s}{\exp(k_{\mathbb{B}^2}(x_p,x_{p+1}))-1}\right)^{1-n}.
\end{equation*}
Moreover, due to $k_{\mathbb{B}^2}(x,y)\leq 2 j_{\mathbb{B}^2}(x,y)$, the preceding inequality leads to
\begin{equation*}
\left(\log \frac{sd(x_p)}{|x_p-x_{p+1}|}\right)^{1-n}\leq \left(\log \frac{s}{\exp(2j_{\mathbb{B}^2}(x_p,x_{p+1}))-1}\right)^{1-n}.
\end{equation*}
Finally, by applying Theorem \ref{Thm-Har con for qr} and utilizing \eqref{inq-jpj}, we derive
\begin{equation*}
\frac{\exp(e^{p+1})}{\exp(e^p)}\leq \exp\left(\frac{A K_I(f)}{c_n}\omega_{n-1}\left(\log \frac{s}{\exp(2\rho_{\mathbb{B}^2}(x_p,x_{p+1}))-1}\right)^{1-n}\right).
\end{equation*}
As $\rho_{\mathbb{B}^2}(x_p,x_{p+1})=1$, the right-hand side of the last inequality remains bounded. However, the left-hand side of the same inequality diverges to infinity as $p$ approaches infinity. Consequently, we can infer that the assertion in Theorem \ref{Thm-Har con for qr} loses validity when $\partial fG$ includes isolated points.
\end{remark}
In the following, we shall study the Harnack metric $h_G(x,y)$, where $G$ is a proper subdomain of $\mathbb{R}^n$.
\begin{theorem}
Let $s\in(0,1)$ and $C(s)\geq 1$. (i) If $G$ is a proper subdomain of $\mathbb{R}^n$, then
  \begin{equation*}
h_G(x,y)\leq \left(1+\frac{k_G(x,y)}{2\log(1+s)}\right)\log C(s).
  \end{equation*}
  (ii) If $G=\mathbb{B}^n$ or $G=\mathbb{H}^n$, then we have
  \begin{equation*}
h_G(x,y)\leq \left(1+\frac{\rho_G(x,y)}{\log[(1+s)/(1-s)]}\right)\log C(s).
  \end{equation*}
\end{theorem}
\begin{proof}
(i) Let $u:G\rightarrow (0,\infty)$ be a Harnack function.
By \cite[Lemma 6.23]{hkv} we have
  \begin{equation*}
    \frac{u(x)}{u(y)}\leq C(s)^{1+t}\Leftrightarrow \log \frac{u(x)}{u(y)}\leq (1+t)\log C(s),
  \end{equation*}
where $t=k_G(x,y)/(2\log(1+s))$. The claim is now a direct consequence of the Harnack metric definition. \\
(ii) According to \cite[Lemma 6.23]{hkv}, the proof closely resembles that of part (i), so we skip the details.
\end{proof}
To prove the next results, the following two lemmas will be helpful.
\begin{lemma}\label{lem-b}(\cite[Corollary 1]{b})
    For all $x,y\in\mathbb{B}^n$,
    \begin{equation*}
        h_{\mathbb{B}^n}(x,y)=2\rho_{\mathbb{B}^n}(x,y).
    \end{equation*}
\end{lemma}
\begin{lemma}\label{lem-bs}(\cite[Lemma 2.5]{bs})
    If $x,y\in\mathbb{H}^n$, then
    \begin{equation*}
        h_{\mathbb{H}^n}(x,y)=\rho_{\mathbb{H}^n}(x,y).
    \end{equation*}
\end{lemma}
\begin{theorem}
i) If $f: \mathbb{B}^n \rightarrow f \mathbb{B}^n$ is a non-constant $K$-quasiregular mapping with $f\mathbb{B}^n \subset \mathbb{B}^n$, then the inequality
\begin{equation*}
h_{f\mathbb{B}^n}(f(x),f(y))\leq 2 K(h_{\mathbb{B}^n}(x,y)/2+\log4)
\end{equation*}
holds for all $x, y \in \mathbb{B}^n$.\\
ii) If $f: \mathbb{B}^n \rightarrow f \mathbb{B}^n=\mathbb{B}^n$ is a $K$-quasiconformal mapping, then the inequality
\begin{equation*}
h_{f\mathbb{B}^n}(f(x),f(y))\leq b \max\{h_{\mathbb{B}^n}(x,y), 2^{1-\alpha}h_{\mathbb{B}^n}(x,y)^\alpha\}
\end{equation*}
holds, where $\alpha=K^{1/(1-n)}$ and $b$ is a constant depending on $K$ and $n$. Here, $b$ tends to $1$ as $K$ tends to $1$.
\end{theorem}

\begin{proof}
(i) By \cite[Theorem 16.2 (2)]{hkv} we have
  \begin{equation}\label{inq-p2-Thm-h les 4 n+1 rho}
    \rho_{f\mathbb{B}^n}(f(x),f(y))\leq K(\rho_{\mathbb{B}^n}(x,y)+\log 4)
  \end{equation}
for all $x,y\in \mathbb{B}^n$, where $f:\mathbb{B}^n\rightarrow f\mathbb{B}^n\subset \mathbb{B}^n$ is a $K$-quasiregular mapping. It follows also from Lemma \ref{lem-b} that, for $x,y\in \mathbb{B}^n$
\begin{equation}\label{inq-p1-Thm-h les 4 n+1 rho}
h_{f\mathbb{B}^n}(f(x),f(y))= 2 \rho_{f\mathbb{B}^n}(f(x),f(y)).
\end{equation}
Now, combining \eqref{inq-p1-Thm-h les 4 n+1 rho} and \eqref{inq-p2-Thm-h les 4 n+1 rho} with Lemma \ref{lem-b} gives the desired result.\\
(ii) Let $f: \mathbb{B}^n \rightarrow f \mathbb{B}^n=\mathbb{B}^n$ be a $K$-quasiconformal mapping and $x,y\in\mathbb{B}^n$. Then, by Corollary 18.5 in \cite{hkv} we have:
\begin{equation}\label{Co-HKV-18.5}
\rho_{\mathbb{B}^n}(f(x),f(y))\leq b \max\{\rho_{\mathbb{B}^n}(x,y), \rho_{\mathbb{B}^n}(x,y)^\alpha\},
\end{equation}
where $\alpha=K^{1/(1-n)}$ and $b$ is a constant depending on $K$ and $n$. Now, by \eqref{Co-HKV-18.5}, and using Lemma \ref{lem-b}, the conclusion is obtained.
\end{proof}
\begin{theorem}
Let $f:\mathbb{H}^n\rightarrow \mathbb{H}^n$ be a non-constant $K$-quasiregular mapping such that $f\mathbb{H}^n\subset \mathbb{H}^n$. Then
\begin{equation*}
h_{f\mathbb{H}^n}(f(x),f(y))\leq K(h_{\mathbb{H}^n}(x,y)+\log 4),
\end{equation*}
where $K\geq 1$.
\end{theorem}
\begin{proof}
If $f:\mathbb{H}^n\rightarrow \mathbb{H}^n$ is a non-constant $K$-quasiregular mapping such that $f\mathbb{H}^n\subset \mathbb{H}^n$, then by \cite[Theorem 16.2 (2)]{hkv}, we have
\begin{equation}\label{HKV-Thm 16.2}
    \rho_{f\mathbb{H}^n}(f(x),f(y))\leq K (\rho_{\mathbb{H}^n}(x,y)+\log 4),
\end{equation}
where $K\geq 1$. Also, by Lemma \ref{lem-bs}, for all $x,y\in f\mathbb{H}^n\subset \mathbb{H}^n$, we have:
  \begin{equation}\label{inq-Herron}
    h_{f\mathbb{H}^n}(f(x),f(y))=  \rho_{f\mathbb{H}^n}(f(x),f(y)).
  \end{equation}
The result now follows from \eqref{HKV-Thm 16.2}-\eqref{inq-Herron}, and Lemma \ref{lem-bs}. The proof is now complete.
\end{proof}
For $r\in(0,1)$ and $K\in[1,\infty)$, the function $\varphi_K:[0,1]\rightarrow [0,1]$ is defined as follows:
\begin{equation*}
\varphi_K(r)=\mu^{-1}\left(\frac{\mu(r)}{K}\right),
\quad \varphi_K(0)=0; \varphi_K(1)=1,
\end{equation*}
where $\mu:(0,1)\rightarrow(0,\infty)$ is a decreasing homeomorphism given by
\begin{equation*}
\mu(r)=\frac{\pi}{2}\frac{\K(\sqrt{1-r^2})}{\K(r)}, \quad {\rm with }\quad \K(r)=\frac{\pi}{2}F\left(\frac{1}{2},\frac{1}{2};1;r^2\right),
\end{equation*}
and $F$ represents the Gaussian hypergeometric function. For additional information about the function $\varphi_K(r)$ and its approximation, readers are encouraged to consult \cite{krv-Landen}.
\begin{theorem}\label{Thm-Har con n=2}
If $f:\mathbb{B}^2\rightarrow\mathbb{B}^2$ is a non-constant $K$-quasiregular mapping, then
  \begin{equation*}
    h_{\mathbb{B}^2}(f(x),f(y))\leq c(K) \max\{h_{\mathbb{B}^2}(x,y), 2^{1-1/K}h_{\mathbb{B}^2}(x,y)^{1/K}\}
  \end{equation*}
for all $x,y\in \mathbb{B}^2$, where $c(K)=2{\rm arth}(\varphi_K({\rm th}(1/2)))$. In particular, $c(1)=1$.
\end{theorem}
\begin{proof}
Let $f:\mathbb{B}^2\rightarrow\mathbb{B}^2$ be a non-constant $K$-quasiregular mapping. Then, by Theorem \cite[Theorem 16.39]{hkv}, we have:
\begin{equation}\label{HKV-Thm 16.39}
        \rho_{\mathbb{B}^2}(f(x),f(y))\leq c(K) \max\{\rho_{\mathbb{B}^2}(x,y), \rho_{\mathbb{B}^2}(x,y)^{1/K}\}
\end{equation}
for all $x,y\in \mathbb{B}^2$, where $c(K)=2{\rm arth}(\varphi_K({\rm th}(1/2)))$.
The desired assertion can be obtained by utilizing Lemma \ref{lem-b} and inequality \eqref{HKV-Thm 16.39}.
\end{proof}

\section{Harmonic Schwarz lemma}\label{sec4-inequ}
This section first generalizes the Schwarz lemma for harmonic functions in the complex plane utilizing the Poisson integral formula. Then, it improves the Schwarz-Pick estimate for a real-valued harmonic function. First, we recall that the classical Schwarz lemma states that if $u:\mathbb{B}^2\rightarrow \mathbb{B}^2$ is a holomorphic function with $u(0)=0$, then
\begin{itemize}
    \item $|u(z)|\leq |z|$ for all $z\in \mathbb{B}^2$;
    \item $|u'(0)|\leq 1$.
\end{itemize}

Heinz (see \cite{hz}) has obtained an improvement of the classical Schwarz lemma for a complex-valued harmonic function, see Lemma \ref{lem-sch} below. A complex-valued function $f:G \to \mathbb{C}$, where $f= u + i v$ is said to be harmonic if both $u:G \to \mathbb{R}$ and $v:G \to \mathbb{R}$ are harmonic in the sense defined above.

\begin{lemma}\label{lem-sch}
    Let $u:\mathbb{B}^2\rightarrow \mathbb{B}^2$ be a complex-valued harmonic function with $u(0)=0$. Then
    \begin{equation*}
|u(z)|\leq \frac{4}{\pi}\arctan |z|.
    \end{equation*}
The inequality is sharp for each point $z\in \mathbb{B}^2$.
\end{lemma}
The following Theorem \ref{thm-Poisson} is known as the Poisson integral formula (see, for example, \cite{g}).
\begin{theorem}\label{thm-Poisson}
Let $u$ be a complex-valued function continuous on $\overline{B}^2(a,R)$, $(R>0)$, and harmonic on ${B}^2(a,R)$. Then for $r\in[0,R)$ and $t\in \mathbb{R}$ the following formulas hold:
\begin{equation}\label{u for harmonic}
    u\left(a+r e^{i t}\right)=\frac{1}{2\pi}\int_{-\pi}^{\pi}\frac{R^2-r^2}{R^2+r^2-2rR \cos(t-\theta)}u\left(a+R e^{i \theta}\right){\rm d}\theta
\end{equation}
and
\begin{equation}
u(a)= \frac{1}{2\pi}\int_{-\pi}^{\pi}u\left(a+R e^{i \theta}\right){\rm d}\theta.
\end{equation}
\end{theorem}


Motivated by Lemma \ref{lem-sch} and applying Theorem \ref{thm-Poisson}, we derive the following Theorem \ref{thm-Extension-Sch lemma} which is an extension of the above Schwarz lemma:

\begin{theorem}\label{thm-Extension-Sch lemma}
Let $0< r<R$ and $M>0$. If $u$ is a complex-valued harmonic mapping in the disk $B^2(a,R)$ such that $|u(w)|\leq M$ for all $w\in B^2(a,R)$, then
\begin{equation*}
\left|u(a+z)-\frac{R^2-|z|^2}{R^2+|z|^2}u(a)\right|\leq\frac{2M}{\pi}\arctan\left(\frac{2R|z|}{R^2-|z|^2}\right),\quad z=re^{i t}.
\end{equation*}
The result is sharp.
\end{theorem}
\begin{proof}
Suppose that $0<r<R$. Applying formula \eqref{u for harmonic} for $z=r$, we obtain
\begin{align*}
    &u\left(a+r\right)-\frac{R^2-r^2}{R^2+r^2}u(a)\\ &\quad=\frac{1}{2\pi}\int_{-\pi}^{\pi}
    \left(\frac{R^2-r^2}{R^2+r^2-2rR \cos(\theta)}-\frac{R^2-r^2}{R^2+r^2} \right)u\left(a+Re^{i\theta}\right){\rm d}\theta\\
    &\quad= \frac{rR(R^2-r^2)}{\pi(R^2+r^2)}\int_{-\pi}^{\pi}\frac{\cos(\theta)}{R^2+r^2-2rR \cos(\theta)} u\left(a+Re^{i\theta}\right){\rm d}\theta.
  \end{align*}
By the last equality and the assumption $|u|\leq M$, we obtain
\begin{equation}\label{inq-p01 Sch Lem}
\left|u\left(a+r\right)-\frac{R^2-r^2}{R^2+r^2}u(a)\right|\leq M\frac{rR(R^2-r^2)}{\pi(R^2+r^2)}\int_{-\pi}^{\pi}\frac{|\cos(\theta)|}{R^2+r^2-2rR \cos(\theta)} {\rm d}\theta.
\end{equation}
Now, we calculate the integral
\begin{equation*}
I=\int_{-\pi}^{\pi}\frac{|\cos(\theta)|}{R^2+r^2-2rR \cos(\theta)} {\rm d}\theta.
\end{equation*}
It is easy to check that,
\begin{align*}
I&=\int_{-\pi/2}^{\pi/2}\left(\frac{\cos(\theta)}{R^2+r^2-2rR \cos(\theta)}+\frac{\cos(\theta)}{R^2+r^2+2rR \cos(\theta)}\right) {\rm d}\theta\\
&=2(R^2+r^2)\int_{-\pi/2}^{\pi/2}\frac{\cos(\theta)}{(R^2+r^2)^2-4r^2 R^2 \cos^2(\theta)}{\rm d}\theta\\
&=4(R^2+r^2)\int_{0}^{\pi/2}\frac{\cos(\theta)}{(R^2-r^2)^2+4r^2 R^2 \sin^2(\theta)}{\rm d}\theta\\
&=\frac{2(R^2+r^2)}{rR(R^2-r^2)}\arctan\left(\frac{2rR}{R^2-r^2}\right).
\end{align*}
Thus, from \eqref{inq-p01 Sch Lem} follows that
\begin{equation*}
\left|u\left(a+r\right)-\frac{R^2-r^2}{R^2+r^2}u(a)\right|\leq \frac{2M}{\pi}\arctan\left(\frac{2rR}{R^2-r^2}\right),
\end{equation*}
which implies the desired result. It is easy to see that the result is sharp for the function
\begin{equation*}
    u_0(z)=-\frac{2M}{\pi}\arg \left(\frac{R-z}{R+z}\right)=\frac{2M}{\pi}\arctan\left( \frac{2rR\sin \theta}{R^2-r^2}\right),\quad z=re^{it},
\end{equation*}
or one of its rotations, where $0<r<R$ and $M>0$, completing the proof.
\end{proof}
\begin{remark}
It should be noted that Theorem \ref{thm-Extension-Sch lemma} is also an extension of \cite[Theorem 3.6.1]{p}. Indeed, Pavlovi\'c proved that if $f:\mathbb{B}^2\rightarrow \overline{\mathbb{B}}^2$ is a complex-valued harmonic function, then the following sharp inequality holds:
\begin{equation*}
\left|u(z)-\frac{1-|z|^2}{1+|z|^2}u(0)\right|\leq\frac{4}{\pi}\arctan |z|,\quad z=re^{i t}.
\end{equation*}
\end{remark}
Let $\nabla u$ be the gradient of $u$ at $x$ defined by
\begin{equation*}
\nabla u(x)=(\partial u/\partial x_1,\ldots,\partial u/\partial x_n).
\end{equation*}
In 1989 (see \cite{Col}), Colonna proved the following Schwarz-Pick estimate for complex-valued harmonic functions $u$ from the unit disk $\mathbb{B}^2$ to itself:
\begin{equation*}
    \left|\frac{\partial u(z)}{\partial z}\right|+\left|\frac{\partial u(z)}{\partial \overline{z}}\right|\leq \frac{4}{\pi}\frac{1}{1-|z|^2}, \quad z\in \mathbb{B}^2.
\end{equation*}
If $u$ is a real-valued function, Kalaj and Vuorinen established the above Schwarz-Pick estimate as the following theorem; refer to \cite[Theorem 1.8]{kv} for details.
\begin{theorem}\label{Thm-KV}
Let $u$ be a real harmonic function of the unit disk into $(-1, 1)$.
Then the following sharp inequality holds:
\begin{equation*}
|\nabla u(z)|\leq \frac{4}{\pi}\frac{1-|u(z)|^2}{1-|z|^2}, \quad z\in \mathbb{B}^2.
\end{equation*}
\end{theorem}
In accordance with the findings of Chen \cite[Theorem 1.2]{cn}, the subsequent result has been derived:
\begin{theorem}\label{the-chen}
Let $u$ be a real harmonic mapping of $\mathbb{B}^2$ into the open interval $(-1,1)$. Then
    \begin{equation*}
        |\nabla u(z)|\leq \frac{4}{\pi}\frac{\cos \left(\frac{\pi}{2}u(z)\right)}{1-|z|^2}
    \end{equation*}
holds for $z\in \mathbb{B}^2$. The inequality is sharp for any $z\in \mathbb{B}^2$ and any value of $u(z)$, and the equality occurs for some point in $\mathbb{B}^2$ if and only if $u(z)=(4 {\rm Re}\{\arctan\, f(z)\})/\pi$, $z\in\mathbb{B}^2$ with a M\"obius transformation $f$ of $\mathbb{B}^2$ onto itself.
\end{theorem}

In the subsequent discussion, we aim to expand upon Theorem \ref{Thm-KV} in the following manner: Furthermore, it is worth noting that our extension encompasses the findings presented in Theorem 6.26 of \cite{abr}.
\begin{theorem}\label{thm-new Kalaj V}
Let $\alpha$ and $\beta$ be two real numbers such that $\alpha<\beta$. If $u:\mathbb{B}^2\rightarrow (\alpha,\beta)$ is a real-valued harmonic function, then we have
  \begin{equation*}
    |\nabla u(z)|\leq \frac{2(\beta-\alpha)}{\pi}\frac{1-\frac{4}{(\beta-\alpha)^2}
    \left|u(z)-\frac{\alpha+\beta}{2}\right|^2}{1-|z|^2},\quad z\in \mathbb{B}^2.
  \end{equation*}
The result is sharp.
\end{theorem}
\begin{proof}
Define $v(z)$ as
  \begin{equation*}
    v(z)=\frac{2}{\beta-\alpha}\left(u(z)-\frac{\alpha+\beta}{2}\right),\quad z\in \mathbb{B}^2,
  \end{equation*}
where $u:\mathbb{B}^2\rightarrow (\alpha,\beta)$ is a real valued harmonic function, $\alpha$ and $\beta$ are real numbers such that $\alpha<\beta$. Then it is clear that $v$ is a harmonic function of the unit disk $\mathbb{B}^2$ into $(-1,1)$. Therefore, $v$ satisfies the assumption of Theorem \ref{Thm-KV}. Moreover, we have
  \begin{equation*}
   \frac{2}{\beta-\alpha}|\nabla u|=|\nabla v|\leq \frac{4}{\pi}\frac{1-\frac{4}{(\beta-\alpha)^2}
    \left|u(z)-\frac{\alpha+\beta}{2}\right|^2}{1-|z|^2},\quad z\in \mathbb{B}^2,
  \end{equation*}
which implies the desired result. To show that the result is sharp, we take the harmonic function
\begin{equation*}
\ell(z)=\frac{\alpha+\beta}{2}+\frac{\beta-\alpha}{\pi}\arctan \frac{2y}{1-x^2-y^2},\quad z\in \mathbb{B}^2.
\end{equation*}
It is easy to see $\alpha<\ell(z)<\beta$. A simple calculation yields
\begin{equation*}
    |\nabla \ell(0)|=\frac{2(\beta-\alpha)}{2}=\frac{2(\beta-\alpha)}{2}\cdot
    \frac{1-\frac{4}{(\beta-\alpha)^2}\left|\frac{\alpha+\beta}{2}-\frac{\alpha+\beta}{2}
    \right|^2}{1-0^2},
\end{equation*}
which is the desired conclusion.
\end{proof}
Applying Theorem \ref{the-chen}, we get the following result:
\begin{theorem}
If $u:\mathbb{B}^2\rightarrow(\alpha,\beta)$ is an into harmonic mapping, then
\begin{equation*}
    |\nabla u(z)|\leq \frac{2(\beta-\alpha)}{\pi}\frac{\cos \left(\frac{\pi}{\beta-\alpha}\left(u(z)-\frac{\alpha+\beta}{2}\right)\right)}{1-|z|^2},
\end{equation*}
where $\alpha$ and $\beta$ are real numbers such that $\alpha<\beta$. The result is sharp.
\end{theorem}
\begin{proof}
The proof is the same as the proof of Theorem \ref{thm-new Kalaj V}, therefore, we omit the details.
\end{proof}
We conclude this paper by presenting the following open question:

\vspace{0.5cm}
\noindent
{\bf Open question.} What is the connection between the Harnack metric $h$ and the hyperbolic metric $\rho$ in a simply connected Jordan domain in the complex plane $\mathbb{C}$?

\vspace{0.5cm}
\noindent
{\bf Acknowledgments.}
The author thanks Professor Matti Vuorinen for his encouragement and many useful discussions throughout the writing process. He also thanks the editor and the anonymous reviewer for their insightful comments and valuable suggestions on this paper, which have greatly enhanced its quality and depth.

\vspace{0.5cm}
\noindent
{\bf Funding.}
This research was supported by the Doctoral Programme (EXACTUS) and the Analysis Foundation of the Department of Mathematics and Statistics of the University of Turku.

\bibliographystyle{siamplain}

\end{document}